\RequirePackage{fix-cm}
\documentclass[smallcondensed]{svjour3}     
\smartqed  
\usepackage{mathptmx}      
\usepackage{amsmath}
\usepackage{latexsym}
\usepackage{graphicx}
\usepackage{xhfill}
\newtheorem{algorithm}{Algorithm}

\begin{document}
\newcommand{\xfill}[2][1ex]{{%
  \dimen0=#2\advance\dimen0 by #1
  \leaders\hrule height \dimen0 depth -#1\hfill%
}}
\title{Modified golden ratio algorithms for solving equilibrium problems
}

\titlerunning{An explicit golden ratio algorithm for EPs}        

\author{Dang Van Hieu \and Jean Jacques Strodiot \and Le Dung Muu
}

\authorrunning{Dang Van Hieu \and Jean Jacques Strodiot \and Le Dung Muu} 
\institute{Dang Van Hieu \at
              Applied Analysis Research Group, Faculty of Mathematics and Statistics, Ton Duc Thang University, Ho Chi Minh City, Vietnam\\
              \email{dangvanhieu@tdtu.edu.vn}           
\and
Jean Jacques Strodiot \at
	Department of Mathematics, Namur Institute for Complex Systems, University
of Namur, Namur, Belgium\\
\email{jjstrodiot@fundp.ac.be}		       
\and
           Le Dung Muu \at
              TIMAS, Thang Long University, Ha Noi, Vietnam \\
               \email{ldmuu@math.ac.vn}     
}

\date{Received: date / Accepted: date}

\maketitle

\begin{abstract}
In this paper an explicit algorithm is proposed for solving an equilibrium problem whose associated bifunction is pseudomonotone and satisfies 
a Lipschitz-type condition. Contrary to many algorithms, our algorithm is done without using explicitly the Lipschitz constants of bifunction 
although its convergence is obtained under such that condition. The introduced method is a form of proximal-like method whose steplengths 
are explicitly generated at each iteration without using any linesearch procedure.  First we prove the convergence of the algorithm, and after 
we establish its $R$-linear rate of convergence  under the assumption of strong pseudomonotonicity of the bifunction.  Afterwards  several 
numerical results are displayed to illustrate and to compare the behavior of the new algorithm with other ones.
\end{abstract}
\section{Introduction}\label{intro}
This paper concerns an iterative method for approximating a solution of an equilibrium problem (shortly, EP) in the sense of Blum, Muu and Oettli in 
\cite{BO1994,M1984,MO1992}. This problem is also called the Ky Fan inequality \cite{K1972} due to his early contribution in this field. 
Problem (EP) can be considered a general mathematical model because it unifies in a simple form numerous 
models as optimization problems, variational inequalities, fixed point problems and many others, see for example \cite{FP2002,K2007}. This can be the reason why in recent years problem (EP) has received a lot of attention by some authors both theoretically and algorithmically. The most well-known algorithms 
for solving problem (EP) are the proximal point method \cite{M1999,K2003}, the proximal-like method (extragradient method) \cite{FA1997,QMH2008}, 
the descent method \cite{BCPP13,KP2003}, the linesearch extragradient method \cite{QMH2008}, the projected subgradient method \cite{SS2011}, the golden 
ratio algorithm \cite{V2018} and others \cite{AH2018,BCPP2019,FA1997,HCX2018,H2018MMOR,Hieu2018NUMA,LS2016,SNN2013}.\\[.1in]
The proximal point method is based on the so-called resolvent of bifunction. At each iteration, this method consists in solving a regularized equilibrium 
subproblem depending on 
a parameter. The solutions of these subproblems can converge finitely or asymptotically to some solution of the original problem. 
In this paper, we 
are interested in another well-known kind of methods using optimization subprograms. It is  the extragradient method as developed in 
\cite{FA1997,HCX2018,QMH2008}, where 
 two optimization programs are solved at each iteration. This can be expensive in the cases where the bifunction and/or the feasible set have complicated structures. 
Very recently, a nice and elegant algorithm, named the golden ratio algorithm, has been proposed by Malitsky \cite{M2018} for solving (pseudo) monotone 
variational inequalities in finite dimensional spaces. Unlike extragradient-like algorithms, the golden ratio algorithm \cite{M2018} only requires to compute one 
projection on feasible set and one value of operator at the curren approximation. This algorithm is done in both cases with and without previously konwing the Lipschit constant of operator.\\[.1in]
Recently, motivated by the results of Malitsky \cite{M2018}, Vinh has introduced an algorithm which only uses one optimization program per 
iteration to construct solution approximations for problem (EP), see \cite[Algorithm 3.1]{V2018} for more details. At this stage, it is emphasized that the algorithms 
in \cite{FA1997,HCX2018,QMH2008,V2018} are applied to pseudomonotone (EP) under a Lipschitz-type condition, and that these algorithms explicitly use  a stepsize depending 
on the Lipschitz-type constants of the bifunction. In particular, this means that the 
Lipschitz-type constants must be the input parameters of the algorithms although these constants are often unknown or difficult to estimate. 
In his paper Vinh has presented another algorithm \cite[Algorithm 4.1]{V2018} where the Lipschitz-type constants associated with the bifunction are not 
supposed to be a priory known. The stepsizes are defined explicitly at each iteration in such a way that their sequence is decreasing and not summable. 
Under these new rules, Vinh has established the strong convergence of the iterative sequence generated by his algorithm without its rate of convergence.\\[.1in]
In this paper, motivated and inspired by the aforementioned results, we introduce an iterative algorithm for solving an equilibrium problem involving 
a pseudomonotone and Lipschitz-type bifunction in a finite dimensional space. The algorithm is explicit in the sense that it is done without previously 
knowing the Lipschitz-type constants. The algorithm uses variable stepsizes which are generated at each iteration and are based on some previous iterates. 
No linesearch procedure is required and the convergence of the resulting algorithm is obtained under the assumption of  
pseudomonotonicity of the bifunction. These results improve the ones obtained by Vinh in \cite{V2018}.
Furthermore, in the case when the bifunction is strongly pseudomonotone, we can establish the $R$-linear rate of convergence of the algorithm. Numerical results are reported to demonstrate 
the behavior of the new algorithm and also to compare it with some other algorithms.\\[.1in]
The remainder of this paper is organized as follows: In Sect. \ref{pre} we collect some definitions and preliminary results used in the paper. 
Sect. \ref{main} deals with the description of the new algorithm and its convergence. Finally, in Sect. \ref{example}, several numerical experiments are reported 
to illustrate the behavior of the new algorithm.
\section{Preliminaries}\label{pre}
Let $C$ be a nonempty closed convex subset in $\Re^m$ and $f:C\times C\to \Re$ be a bifunction with $f(x,x)=0$ for all $x\in C$. 
The equilibrium problem (shortly, EP) for the bifunction $f$ on $C$ can be stated as follows:
$$
\mbox{Find}~x^*\in C~\mbox{such that}~f(x^*,y)\ge 0\,~\mbox{for all}~y\in C.
\eqno{\rm (EP)}
$$
For solving this problem, we need to recall some concepts of monotonicity of a bifunction, see \cite{BO1994,MO1992} for more details.
A bifunction $f:C\times C\to \Re$ is said to be:\\[.1in]
{\rm (a)}\,  \textit{strongly monotone} on $C$ if there exists a constant $\gamma>0$ such that
$$
f(x,y)+f(y,x)\le -\gamma\, \|x-y\|^2\,~\mbox{for all}~ x,y\in C;
$$
{\rm (b)}\, \textit{monotone} on $C$ if
$$
f(x,y)+f(y,x)\le 0\,~\mbox{for all}~ x,y\in C;
$$
{\rm (c)}\, \textit{pseudomonotone} on $C$ if
$$
f(x,y)\ge 0 \ \Longrightarrow\ f(y,x)\le 0\,~\mbox{for all}~x,y\in C;
$$
{\rm (d)}\, \textit{strongly pseudomonotone} on $C$ if there exists a constant $\gamma>0$ such that
$$
 f(x,y)\ge 0 \ \Longrightarrow\  f(y,x) \le-\gamma\, \|x-y\|^2\,~\mbox{for all}~ x,y\in C.
 $$
\vskip 2mm

\noindent From the above definitions, it is easy to see that the following implications hold:
$$
 {\rm (a)}\,\Longrightarrow\, {\rm (b)}\,\Longrightarrow\, {\rm (c)}\,~{\rm and}\,~{\rm (a)}\,\Longrightarrow\, {\rm (d)}\,\Longrightarrow \,{\rm (c)}.
$$
We say that the bifunction $f$ satisfies a Lipschitz-type condition if there exist $c_1>0,c_2>0$ such that
$$
 f(x,y) + f(y,z) \geq f(x,z) - c_1\|x-y\|^2 - c_2\|y-z\|^2\, ~ \mbox{for all}~ x,~y,~z \in C.
 $$
Let $g:C\to \Re$ be a proper, lower semicontinuous, convex function and let $\lambda >0$. The proximal operator ${\rm prox}_{\lambda g}$ associated with $g$ and $\lambda>0$ 
is defined by 
$$
{\rm prox}_{\lambda g}(z)=\arg\min \left\{\lambda g(x)+\frac{1}{2}\|x-z\|^2:x\in C\right\},~z\in \Re^m.
$$
The following lemma gives an important property of the proximal mapping (see \cite{BC2011} for more details)
\begin{lemma}\label{prox}
Let $z \in \Re^m$. Then
$\bar{x}={\rm prog}_{\lambda g}(z) \Leftrightarrow \left\langle \bar{x}-z,x-\bar{x}\right\rangle \ge \lambda \left(g(\bar{x})-g(x)\right),~\forall x\in C.$
\end{lemma}
\begin{remark}\label{rem1}
From Lemma \ref{prox}, it is easy to show that if $x={\rm prox}_{\lambda g}(x)$ then 
$$x\in {\rm arg}\min\left\{g(y):y\in C\right\}:=\left\{x\in C: g(x)=\min_{y\in C}g(y)\right\}.$$
\end{remark}
The next result is also valid in any Hilbert space, see, e.g., in \cite[Corollary 2.14]{BC2011}.
\begin{lemma}\label{eq} For all $x,~y\in \Re^m$ and $\alpha\in \Re$, the following equality always holds
$$
\|\alpha x+(1-\alpha)y\|^2=\alpha \|x\|^2+(1-\alpha)\|y\|^2-\alpha(1-\alpha)\|x-y\|^2.
$$
\end{lemma}
\section{Explicit Golden Ratio Algorithm}\label{main}
In this section, we introduce an explicit golden ratio algorithm (EGRA) for solving the pseudomonotone equilibrium problem (EP) under a Lipschitz-type condition. 
Although, for the convergence of the algorithm, we assume that the bifunction $f$ satisfies Lipschitz-type condition, but the prior knowledge or even an estimate of 
Lipschitz-type constants is not necessary to be required. This  is particularly interesting when these constants are unknown or difficult to approximate.\\[.1in]
For displaying the new algorithm  and for the sake of simplicity, we use the notation 
$[t]_+=\max \left\{0,t\right\}$ and adopt the convention $\frac{0}{0}=+\infty$. Now our algorithm can be displayed in details as follows:\\
\noindent\rule{12.1cm}{0.4pt} 
\begin{algorithm}[EGRA for Equilibrium Problem].\label{alg1}\\
\noindent\rule{12.1cm}{0.4pt}\\
\textbf{Initialization:} Set $\varphi=\frac{1}{2}\,(\sqrt{5}+1)$. Choose $\bar{x}_{-1}\in \Re^m,~x_{-1},~x_0\in C$, $\lambda_0>0$ and $\mu\in \left(0,\frac{1}{2}\varphi\right)$.\\
\noindent\rule{12.1cm}{0.4pt}\\
\textbf{Iterative Steps:} Assume that $\bar{x}_{n-1}\in \Re^m, ~x_{n-1},~x_n\in C$ and $\lambda_n$ are known and calculate $x_{n+1}$ and $\lambda_{n+1}$ as follows:
$$
\left \{
\begin{array}{ll}
\bar{x}_n=\frac{(\varphi-1)x_n+\bar{x}_{n-1}}{\varphi},\\
x_{n+1}=\mbox{\rm prox}_{\lambda_n f(x_n,.)}(\bar{x}_n)\\
\lambda_{n+1}=\min \left\{\lambda_n,\frac{\mu (\|x_{n-1}-x_n\|^2+\|x_n-x_{n+1}\|^2)}{2\left[f(x_{n-1},x_{n+1})-f(x_{n-1},x_n)-f(x_n,x_{n+1})\right]_+}\right\}.
\end{array}
\right.
$$
\end{algorithm}
\noindent\rule{12.1cm}{0.4pt}
\begin{remark}\label{rem2}
For Algorithm \ref{alg1}, we can use the following stopping criterion:
\begin{center}
if $x_{n+1}=x_n=\bar{x}_n$ then stop: the iterate $x_n$ is a solution of problem (EP).
\end{center} 
In fact, this criterion comes from the definition of $x_{n+1}$ and from the Remark \ref{rem1}.
\end{remark}
\begin{remark}
 Contrary to the algorithm in \cite[Algorithm 3.1]{V2018} Algorithm \ref{alg1} does not require to know the values (even, the estimates) of the two Lipschitz-type 
constants associated to $f$ and is also without any linesearch procedure. The sequence of the stepsizes $\lambda_n$ can be suitably updated at each iteration by some 
cheap computations. 
\end{remark}

As Remark \ref{rem3} below, the sequence of stepsizes $\left\{\lambda_n\right\}$ generated by Algorithm \ref{alg1} is separated from $0$. Then, our algorithm 
can be more attractive than an algorithm with diminishing stepsizes or with a linesearch procedure which requires many computations over iteration with some 
stopping criterion, and of course this is time-consuming. 
\subsection{The convergence of EGRA}
In this part, we establish the convergence of algorithm EGRA. For that purpose, we consider the following standard assumptions imposed on the bifunction $f$:
\vskip 1mm

{\rm (A1)}\,  $f(x,x)=0$ for all $x\in C$ and $f$ is pseudomonotone on $C$;

{\rm (A2)}\, $f$ satisfies the Lipschitz-type condition on $C$;

{\rm (A3)}\, $f(.,y)$ is upper semicontinuous for each $y\in C$;

{\rm (A4)}\, $f(x,\cdot)$ is convex and subdifferentiable on $C$ for each $x\in C$.
\vskip 2mm
\begin{remark}\label{rem3}
Let $n$ be fixed. Since $f$ satisfies the Lipschitz-type condition, we obtain for all~$n$
\begin{eqnarray*}
f(x_{n-1},x_{n+1})-f(x_{n-1},x_n)-f(x_n,x_{n+1})&\le&c_1 \|x_{n-1}-x_n\|^2+c_2 \|x_n-x_{n+1}\|^2\\ 
&\le&\max\left\{c_1,c_2\right\}\left[ \|x_{n-1}-x_n\|^2+\|x_n-x_{n+1}\|^2\right].
\end{eqnarray*}
Hence, when $f(x_{n-1},x_{n+1})-f(x_{n-1},x_n)-f(x_n,x_{n+1})>0$, we can deduce that
$$ \min\bigg\{\lambda_n, \frac{\mu}{2 \max\{c_1,c_2\}}\bigg\} \leq \lambda_{n+1} \leq \lambda_n.$$
Since, by convention, $\lambda_{n+1}=\lambda_n$ when $f(x_{n-1},x_{n+1})-f(x_{n-1},x_n)-f(x_n,x_{n+1})\leq 0$, we obtain that the sequence $\{\lambda_n\}$ is nonincreasing and bounded below by $ \min\bigg\{\lambda_0, \frac{\mu}{2 \max\{c_1,c_2\}}\bigg\}$. Hence $\lambda_n \to \lambda >0$.
\end{remark}
\begin{remark}
A characteristic of the constant $\varphi=\frac{1}{2}\,(\sqrt{5}+1)$ in Algorithm \ref{alg1} is that 
$$
1+\frac{1}{\varphi}=\varphi\qquad \mbox{or}\qquad \varphi^2-\varphi-1=0.
$$
This is technically used in our analysis.
\end{remark}
We have the following convergence result.
\begin{theorem}\label{theo1}
Under assumptions {\rm (A1)-(A4)}, the sequence $\left\{x_n\right\}$ generated by Algorithm \ref{alg1} converges to some solution of problem (EP).
\end{theorem}
\begin{proof}
Lemma \ref{prox} and the definition of $x_{n+1}$ ensure that 
\begin{equation}\label{eq:1}
\left\langle \bar{x}_n-x_{n+1}, x-x_{n+1}\right\rangle \le \lambda_n \left( f(x_n,x)-f(x_n,x_{n+1})\right), ~\forall x\in C
\end{equation}
which, by multiplying both sides by $2$, follows that
\begin{equation}\label{eq:2}
2\left\langle \bar{x}_n-x_{n+1}, x-x_{n+1}\right\rangle \le 2\lambda_n \left( f(x_n,x)-f(x_n,x_{n+1})\right).
\end{equation}
Using the equality $2 \left\langle a,b\right\rangle=\|a\|^2+\|b\|^2-\|a-b\|^2$ for $a=\bar{x}_n-x_{n+1}$ and $b=x-x_{n+1}$, we obtain from the relation (\ref{eq:2}) that
\begin{eqnarray}
\|\bar{x}_n-x_{n+1}\|^2+\|x_{n+1}-x\|^2-\|\bar{x}_n-x\|^2\le 2\lambda_n \left( f(x_n,x)-f(x_n,x_{n+1})\right), ~\forall x\in C. \label{eq:3}
\end{eqnarray}
Similarly, it follows from relation (\ref{eq:1}) with $n:=n-1$ that
\begin{equation}\label{eq:4}
\left\langle \bar{x}_{n-1}-x_n, x-x_n\right\rangle \le \lambda_{n-1} \left( f(x_{n-1},x)-f(x_{n-1},x_n)\right), ~\forall x\in C,
\end{equation}
and, with $x=x_{n+1}$, that
\begin{equation}\label{eq:5}
\left\langle \bar{x}_{n-1}-x_n, x_{n+1}-x_n\right\rangle \le \lambda_{n-1} \left( f(x_{n-1},x_{n+1})-f(x_{n-1},x_n)\right).
\end{equation}
Since, by definition of $\bar{x}_n$, we have $\bar{x}_{n-1}-x_n=\varphi (\bar{x}_n-x_n)$, relation (\ref{eq:5}) implies that 
$$
\varphi\left\langle \bar{x}_n-x_n, x_{n+1}-x_n\right\rangle \le \lambda_{n-1} \left( f(x_{n-1},x_{n+1})-f(x_{n-1},x_n)\right).
$$
Now, multiplying both sides of the last inequality by $\frac{2\lambda_n}{\lambda_{n-1}}>0$, we obtain
\begin{equation}\label{eq:6}
\frac{\varphi\lambda_n}{\lambda_{n-1}}2\left\langle \bar{x}_n-x_n, x_{n+1}-x_n\right\rangle \le 2\lambda_n \left( f(x_{n-1},x_{n+1})-f(x_{n-1},x_n)\right).
\end{equation}
Thus, from relation (\ref{eq:6}) and the equality
$$2\left\langle \bar{x}_n-x_n, x_{n+1}-x_n\right\rangle=\|\bar{x}_n-x_n\|^2+ \|x_{n+1}-x_n\|^2- \|x_{n+1}-\bar{x}_n\|^2,$$
we deduce the following inequality,
\begin{equation}\label{eq:7}
\frac{\varphi\lambda_n}{\lambda_{n-1}} \left[\|\bar{x}_n-x_n\|^2+ \|x_{n+1}-x_n\|^2- \|x_{n+1}-\bar{x}_n\|^2\right] \le 2\lambda_n \left( f(x_{n-1},x_{n+1})-f(x_{n-1},x_n)\right).
\end{equation}
Summing up both sides of relations (\ref{eq:3}) and (\ref{eq:7}) and using the definition of $\lambda_{n+1}$, we obtain
\begin{eqnarray}
&\|x_{n+1}-x\|^2-\|\bar{x}_n-x\|^2+(1-\frac{\varphi\lambda_n}{\lambda_{n-1}})\|x_{n+1}-\bar{x}_n\|^2+\frac{\varphi\lambda_n}{\lambda_{n-1}} \|\bar{x}_n-x_n\|^2
+\frac{\varphi\lambda_n}{\lambda_{n-1}} \|x_{n+1}-x_n\|^2\nonumber\\ 
&\le 2\lambda_n f(x_n,x)- 2\lambda_n \left[f(x_{n-1},x_n) + f(x_n,x_{n+1}) - f(x_{n-1},x_{n+1})\right]\nonumber\\ 
&\le 2\lambda_n f(x_n,x)+ \frac{\mu\lambda_n}{\lambda_{n+1}} \left(\|x_{n-1}-x_n\|^2 +  \|x_n - x_{n+1}\|^2\right),\label{eq:8}
\end{eqnarray}
which can be rewritten as
\begin{eqnarray}
\|x_{n+1}-x\|^2&-&\|\bar{x}_n-x\|^2+(1-\frac{\varphi\lambda_n}{\lambda_{n-1}})\|x_{n+1}-\bar{x}_n\|^2+\frac{\varphi\lambda_n}{\lambda_{n-1}} \|\bar{x}_n-x_n\|^2\nonumber\\
&&\hspace*{4cm}+\left(\frac{\varphi\lambda_n}{\lambda_{n-1}}-\frac{\mu\lambda_n}{\lambda_{n+1}}\right) \|x_{n+1}-x_n\|^2\nonumber\\ 
&\le& 2\lambda_n f(x_n,x)+ \frac{\mu\lambda_n}{\lambda_{n+1}} \|x_{n-1}-x_n\|^2,~\forall x\in C.\label{eq:8*}
\end{eqnarray}
Thus
\begin{eqnarray}
&&\|x_{n+1}-x\|^2+\left(\frac{\varphi\lambda_n}{\lambda_{n-1}}-\frac{\mu\lambda_n}{\lambda_{n+1}}\right) \|x_{n+1}-x_n\|^2\le \|\bar{x}_n-x\|^2+\frac{\mu\lambda_n}{\lambda_{n+1}} \|x_{n-1}-x_n\|^2\nonumber\\ 
&&\hspace*{2.5cm}-(1-\frac{\varphi\lambda_n}{\lambda_{n-1}})\|x_{n+1}-\bar{x}_n\|^2-\frac{\varphi\lambda_n}{\lambda_{n-1}} \|\bar{x}_n-x_n\|^2+2\lambda_n f(x_n,x).\label{eq:9}
\end{eqnarray}
Since $\bar{x}_{n+1}=\frac{(\varphi-1)x_{n+1}+\bar{x}_{n}}{\varphi}$, we obtain immediately that
$$
x_{n+1}=\frac{\varphi}{\varphi-1}\bar{x}_{n+1}-\frac{1}{\varphi-1}\bar{x_n}=(1+\frac{1}{\varphi-1})\bar{x}_{n+1}-\frac{1}{\varphi-1}\bar{x_n}.
$$
Applying Lemma \ref{eq}, we come to the following equality,
\begin{eqnarray}\label{eq:10}
\|x_{n+1}-x\|^2&=&\frac{\varphi}{\varphi-1}\|\bar{x}_{n+1}-x\|^2-\frac{1}{\varphi-1}\|\bar{x}_n-x\|^2+\frac{\varphi}{(\varphi-1)^2}\|\bar{x}_{n+1}-\bar{x}_n\|^2\nonumber\\ 
&=&\frac{\varphi}{\varphi-1}\|\bar{x}_{n+1}-x\|^2-\frac{1}{\varphi-1}\|\bar{x}_n-x\|^2+\frac{1}{\varphi}\|x_{n+1}-\bar{x}_n\|^2, 
\end{eqnarray}
where the last equality follows from the fact 
$$\|\bar{x}_{n+1}-\bar{x}_n\|^2=\left\|\frac{(\varphi-1)x_{n+1}+\bar{x}_{n}}{\varphi} - \bar{x}_n \right\|^2=\frac{(\varphi-1)^2}{\varphi^2}||x_{n+1}-\bar{x}_n||^2.$$
Combining relations (\ref{eq:9}) and (\ref{eq:10}), we obtain
\begin{eqnarray}
&\frac{\varphi}{\varphi-1}\|\bar{x}_{n+1}-x\|^2+\left(\frac{\varphi\lambda_n}{\lambda_{n-1}}-\frac{\mu\lambda_n}{\lambda_{n+1}}\right) \|x_n - x_{n+1}\|^2\le \frac{\varphi}{\varphi-1} \|\bar{x}_n-x\|^2+ \frac{\mu\lambda_n}{\lambda_{n+1}} \|x_{n-1}-x_n\|^2\nonumber\\ 
&\hspace*{1cm}-(1+\frac{1}{\varphi}-\frac{\varphi\lambda_n}{\lambda_{n-1}})\|x_{n+1}-\bar{x}_n\|^2-\frac{\varphi\lambda_n}{\lambda_{n-1}} \|\bar{x}_n-x_n\|^2+2\lambda_n f(x_n,x),~\forall x\in C.\label{eq:11}
\end{eqnarray}
Note that $\left\{\lambda_n\right\}$ is non-increasing, i.e., $\lambda_n\le \lambda_{n-1}$ for all $n\ge 1$. Then
\begin{equation}\label{e12}
1+\frac{1}{\varphi}-\frac{\varphi\lambda_n}{\lambda_{n-1}}\ge 1+\frac{1}{\varphi}-\frac{\varphi\lambda_{n-1}}{\lambda_{n-1}}=1+\frac{1}{\varphi}-\varphi=0,~\forall n\ge 1.
\end{equation}
Moreover, since $\lambda_n\to \lambda>0$ and $0<\mu<\frac{1}{2}\varphi$, we obtain 
\begin{equation*}
\lim_{n\to\infty}\left(\frac{\varphi\lambda_n}{\lambda_{n-1}}-\frac{\mu\lambda_n}{\lambda_{n+1}}-\frac{\mu\lambda_{n+1}}{\lambda_{n+2}}\right)=\varphi-2\mu>0.
\end{equation*}
Hence, there exists $n_0\ge 1$ such that 
$$
\frac{\varphi\lambda_n}{\lambda_{n-1}}-\frac{\mu\lambda_n}{\lambda_{n+1}}-\frac{\mu\lambda_{n+1}}{\lambda_{n+2}}>0,~\forall n\ge n_0
$$
i.e., 
\begin{equation}\label{e13}
\frac{\varphi\lambda_n}{\lambda_{n-1}}-\frac{\mu\lambda_n}{\lambda_{n+1}}>\frac{\mu\lambda_{n+1}}{\lambda_{n+2}},~\forall n\ge n_0.
\end{equation}
Combining the relations (\ref{eq:11})-(\ref{e13}), we obtain for all $x\in C$ and $n\ge n_0$ that
\begin{eqnarray}
\frac{\varphi}{\varphi-1}\|\bar{x}_{n+1}-x\|^2+\frac{\mu\lambda_{n+1}}{\lambda_{n+2}} \|x_n - x_{n+1}\|^2&\le &\frac{\varphi}{\varphi-1} \|\bar{x}_n-x\|^2+ \frac{\mu\lambda_n}{\lambda_{n+1}} \|x_{n-1}-x_n\|^2\nonumber\\ 
&&-\frac{\varphi\lambda_n}{\lambda_{n-1}} \|\bar{x}_n-x_n\|^2+2\lambda_n f(x_n,x).\label{e14}
\end{eqnarray}
Note that for each $x^*\in EP(f,C)$, we have that $f(x^*,x_n)\ge 0$ because $x_n\in C$. Hence, from the pseudomonotonicity of $f$, we derive  $f(x_n,x^*)\le 0$. 
Now, using relation (\ref{e14}) for $x=x^*\in C$ and setting
$$a_n=\frac{\varphi}{\varphi-1} \|\bar{x}_n-x^*\|^2+ \frac{\mu\lambda_n}{\lambda_{n+1}} \|x_{n-1}-x_n\|^2,$$
$$ b_n=\frac{\varphi\lambda_n}{\lambda_{n-1}} \|\bar{x}_n-x_n\|^2,$$
we deduce that
\begin{equation}\label{e15}
a_{n+1}\le a_n-b_n, ~\forall n\ge n_0.
\end{equation}
Thus, the limit of $\left\{a_n\right\}_{n\ge n_0}$ exists and $\lim\limits_{n\to\infty}b_n=0$. Hence, the sequences $\left\{\bar{x}_n\right\}$ and $\left\{x_n\right\}$ are bounded. Moreover, from the definition of $b_n$ and $\lambda_n\to\lambda>0$, we obtain
\begin{equation}\label{e16}
\lim_{n\to\infty}\|\bar{x}_n-x_n\|^2=0.
\end{equation}
Consequently, since $\bar{x}_{n-1}-x_n=\varphi (\bar{x}_n-x_n)$, we get 
\begin{equation}\label{e17}
\lim_{n\to\infty}\|\bar{x}_{n-1}-x_n\|^2=0.
\end{equation}
From the relations (\ref{e16}) and (\ref{e17}), we have 
\begin{equation}\label{e18}
\lim_{n\to\infty}\|\bar{x}_{n}-\bar{x}_{n-1}\|^2=0.
\end{equation}
Also, from (\ref{e17}), we have that $\|\bar{x}_{n}-x_{n+1}\|^2\to 0$ which together with (\ref{e16}) implies that 
\begin{equation}\label{e19}
\lim_{n\to\infty}\|x_{n+1}-x_n\|^2=0.
\end{equation}
Now, assume that $p$ is a cluster point of the sequence $\left\{x_n\right\}$, i.e., there exists a subsequence of $\left\{x_n\right\}$, denoted by $\left\{x_m\right\}$, which
converges to $p$. We will prove that $p\in EP(f,C)$. Indeed, it follows from relation (\ref{e14}) that 
\begin{eqnarray}
f(x_n,x)&\ge& \frac{\varphi}{2\lambda_n(\varphi-1)} \left(\|\bar{x}_{n+1}-x\|^2-\|\bar{x}_n-x\|^2\right)\nonumber\\ 
&&\hspace*{2cm}- \frac{\mu}{2\lambda_{n+1}} \|x_{n-1}-x_n\|^2+\frac{\mu\lambda_{n+1}}{2\lambda_n\lambda_{n+2}} \|x_n - x_{n+1}\|^2.\label{e20}
\end{eqnarray}
Passing to the limit in the last inequality as $n=m\to \infty$ and using the upper semicontinuity of $f(.,x)$, the relation (\ref{e19}), and the limit $\lambda_n\to\lambda$, we obtain 
\begin{eqnarray}
f(p,x)\ge \limsup_{m\to\infty}f(x_m,x)\ge \frac{\varphi}{2\lambda(\varphi-1)}\lim_{m\to \infty} \left(\|\bar{x}_{m+1}-x\|^2-\|\bar{x}_m-x\|^2\right),~\forall x\in C.
\label{e20}
\end{eqnarray}
Note that from the relations (\ref{e16}) and (\ref{e18}), we also obtain that $\bar{x}_m,~\bar{x}_{m+1} \to p$ as $m\to\infty$. Thus
\begin{eqnarray}
\lim_{m\to\infty}\left|\|\bar{x}_m-x\|^2-\|\bar{x}_{m+1}-x\|^2\right|=0,~\forall x\in C.\label{e22}
\end{eqnarray}
Combining the relations (\ref{e20}) and (\ref{e22}), we get $f(p,x)\ge 0$ for all $x\in C$. Thus $p\in EP(f,C)$. 
To finish the proof, we prove that the whole sequence $\left\{x_n\right\}$ converges to $p$ as $n\to\infty$. Indeed, assume 
that $\left\{x_l\right\}$ is another subsequence of $\left\{x_n\right\}$ converging to $\bar{p}\ne p$. As mentioned above, we have that 
$\bar{p}\in EP(f,C)$. The fact $\lim\limits_{n\to\infty}a_n\in \Re$ and the relation (\ref{e19}) ensure that $\lim\limits_{n\to\infty}\|\bar{x}_n-x^*\|^2\in \Re$, 
and thus $\lim\limits_{n\to\infty}\|x_n-x^*\|^2\in \Re$ for each $x^*\in EP(f,C)$. On the other hand, we have
$$
2 \left\langle x_n,p-\bar{p}\right\rangle = \|x_n-\bar{p}\|^2-\|x_n-p\|^2+\|p\|^2-\|\bar{p}\|^2.
$$
Thus, since $\lim\limits_{n\to\infty}\|x_n-p\|^2\in \Re$ and $\lim\limits_{n\to\infty}\|x_n-\bar{p}\|^2\in \Re$, we obtain that $\lim\limits_{n\to\infty}\left\langle x_n,p-\bar{p}\right\rangle\in \Re$. Setting 
\begin{equation}\label{h31}
\lim\limits_{n\to\infty}\left\langle x_n,p-\bar{p}\right\rangle=M
\end{equation}
and passing to the limit in (\ref{h31}) as $n=k,~l\to \infty$, we obtain 
$$
\left\langle p,p-\bar{p}\right\rangle=\lim\limits_{k\to\infty}\left\langle x_k,p-\bar{p}\right\rangle=M=\lim\limits_{l\to\infty}\left\langle x_l,p-\bar{p}\right\rangle =\left\langle \bar{p},p-\bar{p}\right\rangle.
$$
Thus, $\|p-\bar{p}\|^2=0$ and $\bar{p}=p$. This completes the proof.
\end{proof}
\begin{remark}
The obtained result in Theorem \ref{theo1} can be extended to an infinite dimensional Hilbert space $H$. In that case, hypothesis (A3) is replaced by the sequentially 
weakly upper semicontinuity of $f(.,y)$ on $C$ for each $y\in C$, and the sequence $\left\{x_n\right\}$ generated by Algorithm~\ref{alg1} converges weakly to some solution of problem (EP).
\end{remark}
\subsection{The convergence rate of the Explicit Golden Ratio Algorithm}
In this subsection, we study the convergence rate of algorithm EGRA under the assumption that the bifunction $f$ is strongly 
pseudomonotone (SP) and satisfies the Lipschitz-type condition (LC). Recently, Vinh has proposed in \cite[Algorithm 4.1]{V2018}
a strongly convergent algorithm for solving problem (EP) in a Hilbert space under the conditions (SP) and (LC). His algorithm  
uses a decreasing and non-summable sequence of stepsizes.  However, it is easy to see that such an algorithm  cannot linearly converge. 
Contrary to \cite{V2018}, our algorithm uses an explicit formula to calculate the steplength of the iterates. This strategy will allow us to 
establish the $R$-linear rate of convergence (at least) of  Algorithm \ref{alg1}. Furthermore, in addition to the assumptions (SP) and (LC), 
we will also impose that for each $y\in C$, the function $f(\cdot,y)$ 
is convex, lower semicontinuous and for each $x\in C$, the function $f(x,\cdot)$ is hemicontinuous on $C$. Under these assumptions, 
problem (EP) has a unique solution denoted by $x^\dagger$ (see, e.g., \cite[Proposition 1]{MQ15}).\\[.1in]
Let us recall two fundamental concepts of convergence rate in \cite[Chapter 9]{OR70}. A sequence $\left\{x_n\right\}$ in $H$ converges in norm to $x^*\in H$. 
We say that \\[.1in]
(a) $\left\{x_n\right\}$ converges to $x^*$ with $R$-linear convergence rate if
$$
\limsup_{n\to\infty}||x_n-x^*||^{1/n}<1, 
$$
(b) $\left\{x_n\right\}$ converges to $x^*$ with $Q$-linear convergence rate if there exists $\mu \in (0,1)$ such that
$$
||x_{n+1}-x^*||\le \mu ||x_n-x^*||, 
$$
for all sufficiently large $n$. \\[.1in]
Note that $Q$-linear convergence rate implies $R$-linear convergence rate, see \cite[Section 9.3]{OR70}. The inverse in general is not true. 
Algorithm \ref{alg1} has the following rate of convergence.
\vskip 2mm

\begin{theorem}\label{theo2}
Under the assumptions (SP) and (LC), the sequence $\left\{x_n\right\}$ generated by Algorithm \ref{alg1} converges at least $R$-linearly to 
the unique solution $x^\dagger$ of problem (EP).
\end{theorem}

\begin{proof}
Using the relation (\ref{eq:11}) with $x=x^\dagger \in C$, we obtain
\begin{eqnarray}
&\frac{\varphi}{\varphi-1}\|\bar{x}_{n+1}-x^\dagger\|^2+\left(\frac{\varphi\lambda_n}{\lambda_{n-1}}-\frac{\mu\lambda_n}{\lambda_{n+1}}\right) \|x_n - x_{n+1}\|^2\le \frac{\varphi}{\varphi-1} \|\bar{x}_n-x^\dagger\|^2+ \frac{\mu\lambda_n}{\lambda_{n+1}} \|x_{n-1}-x_n\|^2\nonumber\\ 
&-(1+\frac{1}{\varphi}-\frac{\varphi\lambda_n}{\lambda_{n-1}})\|x_{n+1}-\bar{x}_n\|^2-\frac{\varphi\lambda_n}{\lambda_{n-1}} \|\bar{x}_n-x_n\|^2+2\lambda_n f(x_n,x^\dagger),~\forall n\ge 1.\label{h32}
\end{eqnarray}
Note that $\lambda_n\le \lambda_{n-1}$ for all $n\ge 1$. Thus, $1+\frac{1}{\varphi}-\frac{\varphi\lambda_n}{\lambda_{n-1}}\ge 1+\frac{1}{\varphi}-\varphi=0$. This 
together with the relation (\ref{h32}) and the strong pseudomonotonicity of $f$ implies that 
\begin{eqnarray}
\frac{\varphi}{\varphi-1}\|\bar{x}_{n+1}-x^\dagger\|^2&+&\left(\frac{\varphi\lambda_n}{\lambda_{n-1}}-\frac{\mu\lambda_n}{\lambda_{n+1}}\right) \|x_n - x_{n+1}\|^2\nonumber\\ 
&\le& \frac{\varphi}{\varphi-1} \|\bar{x}_n-x^\dagger\|^2+ \frac{\mu\lambda_n}{\lambda_{n+1}} \|x_{n-1}-x_n\|^2 + 2\lambda_n f(x_n,x^\dagger)\nonumber\\
&\le& \frac{\varphi}{\varphi-1} \|\bar{x}_n-x^\dagger\|^2+ \frac{\mu\lambda_n}{\lambda_{n+1}} \|x_{n-1}-x_n\|^2-2\lambda_n \gamma\, \|x_n-x^\dagger\|^2,\label{h33}
\end{eqnarray}
where $\gamma$ is the modulus of strong pseudomonotonicity of $f$. Note that from the definition of $\bar{x}_n$, we obtain
$
x_n=\frac{\varphi}{\varphi-1}\bar{x}_n-\frac{1}{\varphi-1}\bar{x}_{n-1}.
$
Thus, it follows from Lemma \ref{eq} that
\begin{eqnarray}\label{h34}
\|x_n-x^\dagger\|^2&=&\frac{\varphi}{\varphi-1}\|\bar{x}_n-x^\dagger\|^2-\frac{1}{\varphi-1}\|\bar{x}_{n-1}-x^\dagger\|^2+\frac{\varphi}{(\varphi-1)^2}\|\bar{x}_n-\bar{x}_{n-1}\|^2\nonumber\\
&\ge&\frac{\varphi}{\varphi-1}\|\bar{x}_n-x^\dagger\|^2-\frac{1}{\varphi-1}\|\bar{x}_{n-1}-x^\dagger\|^2.
\end{eqnarray}
From the relations (\ref{h33}) and (\ref{h34}), we see that
\begin{eqnarray}
\frac{\varphi}{\varphi-1}\|\bar{x}_{n+1}-x^\dagger\|^2&+&\left(\frac{\varphi\lambda_n}{\lambda_{n-1}}-\frac{\mu\lambda_n}{\lambda_{n+1}}\right)\|x_{n+1}-x_n\|^2\le\frac{\varphi}{\varphi-1}(1-2\gamma\lambda_n)\|\bar{x}_n-x^\dagger\|^2\nonumber\\
&&+\frac{2\gamma\lambda_n}{\varphi-1}\|\bar{x}_{n-1}-x^\dagger\|^2+\frac{\mu\lambda_n}{\lambda_{n+1}} \|x_n-x_{n-1}\|^2.\label{h34*}
\end{eqnarray}
Let $\rho$ and $\theta$ be two real numbers such that 
\begin{equation}\label{h35}
1<\rho<\frac{\varphi}{\mu}-1\qquad \mbox{and}\qquad 1<\theta<\varphi.
\end{equation}
Thus, from the fact $\lim\limits_{n\to\infty}\lambda_n=\lambda>0$, we have
\begin{eqnarray*}
&&\lim_{n\to\infty}\left(\frac{\varphi\lambda_n}{\lambda_{n-1}}-\frac{\mu\lambda_n}{\lambda_{n+1}}-\frac{\mu\rho\lambda_{n+1}}{\lambda_{n+2}}\right)=\varphi-(1+\rho)\mu>0,\\ 
&&\lim_{n\to\infty}\frac{2\gamma\lambda_n}{\varphi-1}=\frac{2\gamma\lambda}{\varphi-1}<\frac{2\gamma\lambda\theta}{\varphi-1}.
\end{eqnarray*}
These limits imply that there exists $n_0\ge 1$ such that 
\begin{eqnarray}
&&\frac{\varphi\lambda_n}{\lambda_{n-1}}-\frac{\mu\lambda_n}{\lambda_{n+1}}>\frac{\mu\rho\lambda_{n+1}}{\lambda_{n+2}},~\forall n\ge n_0,\label{h36}\\
&&\frac{2\gamma\lambda_n}{\varphi-1}<\frac{2\gamma\lambda\theta}{\varphi-1},~\forall n\ge n_0.\label{h37}
\end{eqnarray}
Moreover, since $\left\{\lambda_n\right\}$ is non-increasing and $\lambda_n\to \lambda$, we obtain that $\lambda_n\ge \lambda_{n+1}\ge 
\lambda_{n+2}\ge \ldots \ge \lambda_\infty= \lambda$. Thus
\begin{equation}\label{h38}
1-2\gamma\lambda_n\le 1-2\gamma\lambda.
\end{equation}
Combining the relation (\ref{h34*}) with the relations (\ref{h36}) - (\ref{h38}), we obtain
 \begin{eqnarray}
\frac{\varphi}{\varphi-1}\|\bar{x}_{n+1}-x^\dagger\|^2&+&\frac{\mu\rho\lambda_{n+1}}{\lambda_{n+2}}\|x_{n+1}-x_n\|^2\le\frac{\varphi}{\varphi-1}(1-2\gamma\lambda)\|\bar{x}_n-x^\dagger\|^2\nonumber\\
&&+\frac{2\gamma\lambda\theta}{\varphi-1}\|\bar{x}_{n-1}-x^\dagger\|^2+\frac{\mu\lambda_n}{\lambda_{n+1}} \|x_n-x_{n-1}\|^2.\label{h39}
\end{eqnarray}
Setting $\Gamma_n=\frac{\varphi}{\varphi-1}\|\bar{x}_n-x^\dagger\|^2$, $\Xi_n=\frac{\mu\rho\lambda_{n}}{\lambda_{n+1}}\|x_{n}-x_{n-1}\|^2$ and 
$\alpha=2\lambda \gamma>0$, the inequality (\ref{h39}) can be shortly rewritten as
\begin{eqnarray}
\Gamma_{n+1}+\Xi_{n+1}&\le&(1-\alpha)\Gamma_n+\frac{\theta\alpha}{\varphi}\Gamma_{n-1}+\frac{\Xi_n}{\rho}.\label{h40}
\end{eqnarray}
Let $r_1>0$ and $r_2>0$. Now, we can rewrite relation (\ref{h40}) in the following form,
\begin{eqnarray}
\Gamma_{n+1}+r_1 \Gamma_n+\Xi_{n+1}&\le& r_2(\Gamma_n+r_1 \Gamma_{n-1})+ \frac{\Xi_n}{\rho}\nonumber\\
&&+(1-\alpha-r_2+r_1)\Gamma_n+(\frac{\theta\alpha}{\varphi} -r_1 r_2) \Gamma_{n-1}.\label{h41}
\end{eqnarray}
Choosing $r_1>0$ and $r_2>0$ such that $1-\alpha-r_2+r_1=0$ and $\frac{\theta\alpha}{\varphi} -r_1 r_2=0$, we obtain, 
by a straightforward computation
$$
r_1=\frac{\alpha-1+\sqrt{(\alpha-1)^2+\frac{4\theta\alpha}{\varphi}}}{2}\quad \mbox{and}\quad r_2=\frac{1-\alpha+\sqrt{(\alpha-1)^2+\frac{4\theta\alpha}{\varphi}}}{2}.
$$
Afterwards we study the behavior of the following function, 
$$
h(t)=\frac{1-t+\sqrt{(t-1)^2+\frac{4\theta t}{\varphi}}}{2}, ~t\in [0,+\infty).
$$
whose derivative is given by 
$$
h'(t)=-\frac{1}{2}+\frac{t-1+\frac{2\theta}{\varphi}}{2\sqrt{(t-1)^2+\frac{4\theta t}{\varphi}}}=
\frac{\frac{4\theta}{\varphi}(\frac{\theta}{\varphi}-1)}{2\sqrt{(t-1)^2+\frac{4\theta t}{\varphi}}\left(t-1+\frac{2\theta}{\varphi}+\sqrt{(t-1)^2+\frac{4\theta t}{\varphi}}\right)}<0
$$
because $1<\theta<\varphi$ (see, the relation (\ref{h35})). Thus, $h(t)$ is non-increasing on $[0,+\infty)$, and $0<r_2=h(\alpha)<h(0)=1$. Next, set 
$\epsilon=\max \left\{\frac{1}{\rho},r_2\right\}$ and note that $\epsilon\in (0,1)$. Then from the relation (\ref{h41}) and  the equalities
$1-\alpha-r_2+r_1=0$ and $\frac{\theta\alpha}{\varphi} -r_1 r_2=0$, we obtain
\begin{eqnarray}
\Gamma_{n+1}+r_1 \Gamma_n+\Xi_{n+1}&\le& r_2(\Gamma_n+r_1 \Gamma_{n-1})+ \frac{\Xi_n}{\rho}\nonumber\\
 &\le&\epsilon (\Gamma_n+r_1 \Gamma_{n-1}+ \Xi_n),~\forall n\ge n_0.\label{h42}
\end{eqnarray}
From the $Q$ - convergence rate for the sequence $\left\{\Gamma_n+r_1 \Gamma_{n-1}+ \Xi_n\right\}$, we obtain the $R$ - convergence rate for 
the sequence $\left\{\Xi_n\right\}$ and thus, from the defintion of $\Xi_n$, for the sequence $\left\{\|x_{n}-x_{n-1}\|\right\}$ because 
the sequence $\left\{\frac{\mu\rho\lambda_{n}}{\lambda_{n+1}}\right\}$ is bounded 
below from $0$. This implies immediately that the sequence $\left\{x_{n}\right\}$ converges $R$ - linearly. This completes the proof.
\end{proof}
\section{Numerical illustrations}\label{example}
In this section, we consider an equilibrium problem which is based on the Nash-Cournot equilibrium model \cite{CKK2004,FP2002}, and we present some 
experiments to describe the numerical behavior of Algorithm \ref{alg1} (EGRA) in comparison with two other algorithms, namely the linesearch 
extragradient method (LEGM) introduced in \cite[Algorithm 2a]{QMH2008} and the ergodic method (ErgM) proposed in \cite{AHT2015}. 
We choose these algorithms for the comparison because they have the same features: the Lipschitz-type constants are not necessarily known. In order 
to demonstrate the computational performance of the algorithms, we use the sequence $D_n=\|x_n-{\rm prox}_{\lambda f(x_n,.)}(x_n)\|^2,~n=0,~1,2,\ldots$ versus 
the number of iterations (\# Iterations)  or the  execution time (Elapsed Time) in seconds. Here $\lambda>0$ and the sequence $\left\{x_n\right\}$ 
is generated by each algorithm. Finally let us observe that $D_n=0$ if and only if $x_n$ is a solution of the problem.\\[.1in]
In a purpose of legibility, we simplify the model as follows: 
Assume that there are $m$ companies that produce a commodity. Let $x$ denote the vector whose entry $x_j$ stands for the quantity  of
 the commodity produced  by company $j$. We suppose that the price $p_j(s)$ is a decreasing affine function of $s$ with $s= \sum_{j=1}^m x_j$, 
i.e., $p_j(s)= \alpha_j - \beta_j s$, where $\alpha_j > 0$, $\beta_j > 0$. Then the profit made by company $j$ is given by $f_j(x)= p_j(s) x_j -
c_j( x_j)$, where $c_j(x_j)$ is the tax and fee for generating $x_j$. Suppose that $C_j=[x_j^{\min},x_j^{\max}]$ is the strategy set of company $j$, then the
strategy set of the model is $C:= C_1\times C_2 \times ...\times C_m$. Actually, each company seeks to  maximize its profit by choosing the
corresponding production level under the presumption that the production of the other companies is a parametric input.
 A commonly used approach to this model is based upon the famous Nash equilibrium concept. We recall  that a point $x^* \in C=C_1\times C_2
\times\cdots\times C_m$ is  an equilibrium point of the model   if
    $$f_j(x^*) \geq f_j(x^*[x_j]), \ \forall x_j \in C_j, \ \forall  j=1,2,\ldots,m,$$
    where the vector $x^*[x_j]$ stands for the vector obtained from
    $x^*$ by replacing $x^*_j$ with $x_j$.
By taking $f(x, y):= \psi(x,y)-\psi(x,x)$ with $\psi(x,y):=  -\sum_{j = 1}^m f_j(x[y_j])$, 
the problem of finding a Nash equilibrium point of the model can be formulated as:
$$\mbox{Find}~x^* \in C~\mbox{such that}~ f(x^*,x) \geq 0 ~\mbox{ for all }x \in C. $$
Furthermore, assume that the tax-fee function $c_j(x_j)$ is increasing and affine for every $j$. This assumption means that both tax and fee  
for producing a unit are increasing as the quantity of the production gets larger. In that case, the bifunction $f$ can be formulated in 
the form 
$$f(x,y)=\left\langle Px+Qy+q,y-x\right\rangle,$$
where $q\in \Re^m$ and $P,~Q$ are two matrices of order $m$ such that $Q$ is symmetric positive semidefinite and $Q-P$ is symmetric 
negative semidefinite. In this case, the bifucntion $f$ is pseudomonotone and satisfies the Lipschitz-type condition. \\[.1in]
We perform the numerical computations in $\Re^m$ with $m=100,~200,~300$; the starting point $x_{-1}=x_0=\bar{x}_{-1}=(1,1,\ldots,1)^T\in \Re^m$; 
the feasible set is a polyhedral convex set given by $C=\left\{x\in \Re^m: Ax\le b\right\}$, where $A$ is a random matrix of size $l\times m$ ($l=10$)
and $b\in\Re^l$ is a random vector such that $x_0\in C$. The control parameters are $\mu=0.45 \varphi$ (for the EGRA); $\lambda_n=\frac{1}{n}$ (for 
the ErgM). The data is generated as follows: All the entries of $q$ are generated randomly 
and uniformly in the interval $[-2,2]$ and the two matrices $P,~Q$ are also 
generated randomly\footnote{We randomly choose $\lambda_{1k}\in [-2,0],~\lambda_{2k}\in [0,2],~ k=1,\ldots,m$. We set $\widehat{Q}_1$, $\widehat{Q}_2$ 
as two diagonal matrixes with eigenvalues $\left\{\lambda_{1k}\right\}_{k=1}^m$ and $\left\{\lambda_{2k}\right\}_{k=1}^m$, respectively. Then, we 
construct a positive semidefinite matrix $Q$ and a negative definite matrix $T$ by using random orthogonal matrixes with $\widehat{Q}_2$ 
and $\widehat{Q}_1$, respectively. Finally, we set $P=Q-T$} such that their conditions hold. All the optimization subproblems are effectively solved 
by the subroutine \textit{quadprog} in Matlab 7.0.\\[.1in]
The numerical results are shown in Figures \ref{fig1} - \ref{fig6} where we can see that the convergence of the EGRA strictly depends on the starting stepsize $\lambda_0$. These results also illustrate that the EGRA works well and has a competitive advantage over the other algorithms.
\begin{figure}[!ht]
\begin{minipage}[b]{0.45\textwidth}
\centering
\includegraphics[height=5cm,width=6cm]{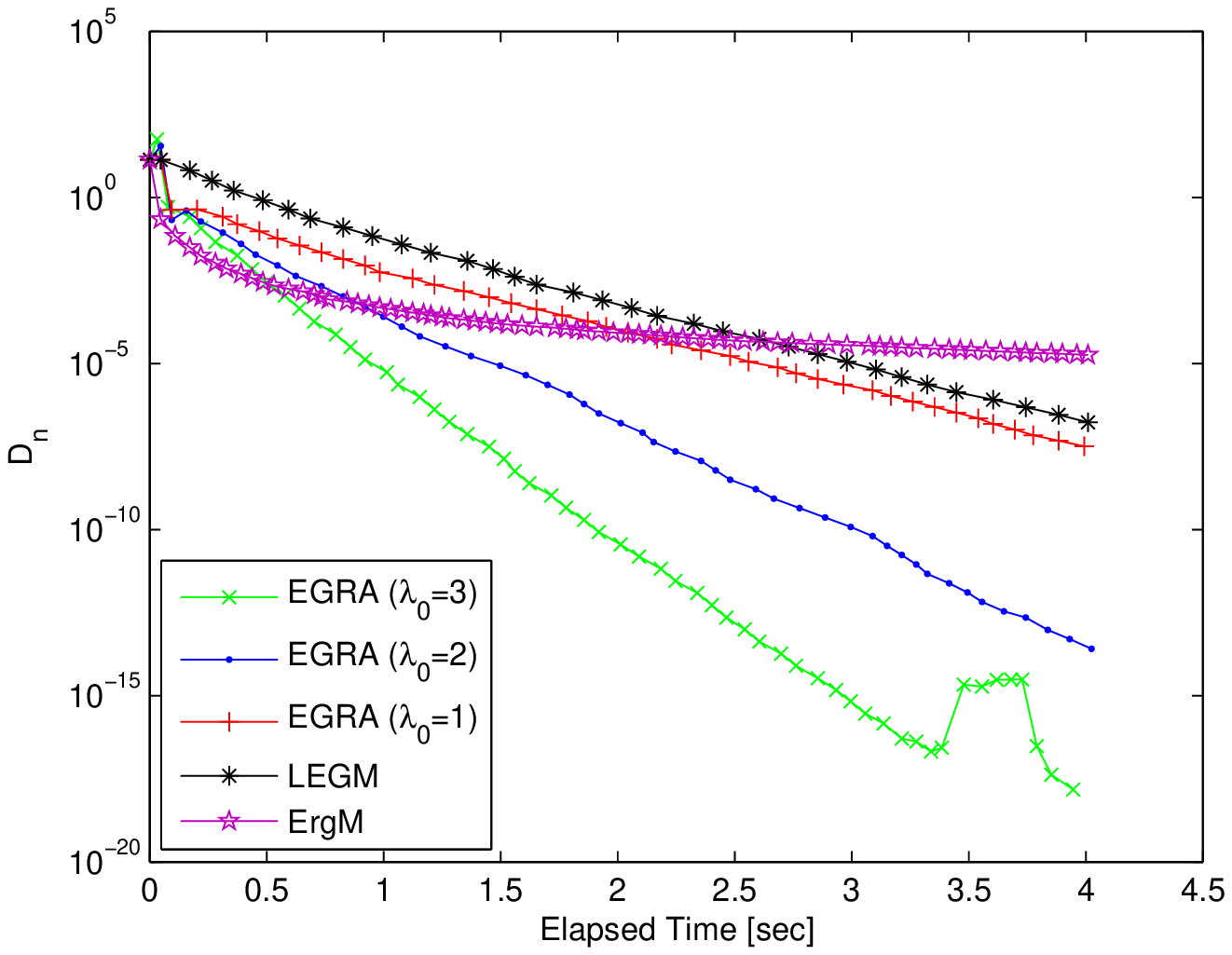}
\caption{$D_n$ and Time (in $\Re^{100}$)}\label{fig1}
\end{minipage}
\hfill
\begin{minipage}[b]{0.45\textwidth}
\centering
\includegraphics[height=5cm,width=6cm]{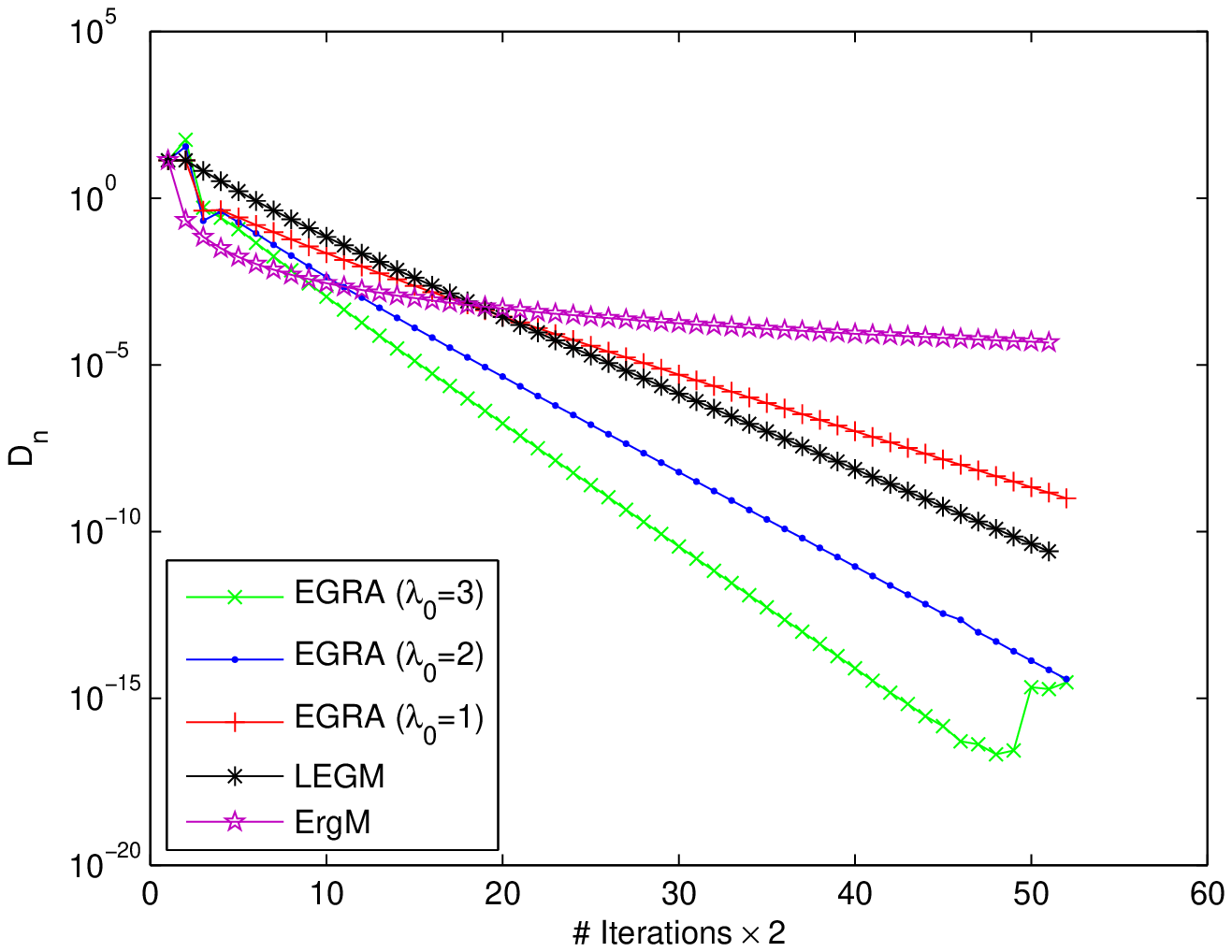}
\caption{$D_n$ and \# Iterations (in $\Re^{100}$)}\label{fig2}
\end{minipage}
\end{figure}

\begin{figure}[!ht]
\begin{minipage}[b]{0.45\textwidth}
\centering
\includegraphics[height=5cm,width=6cm]{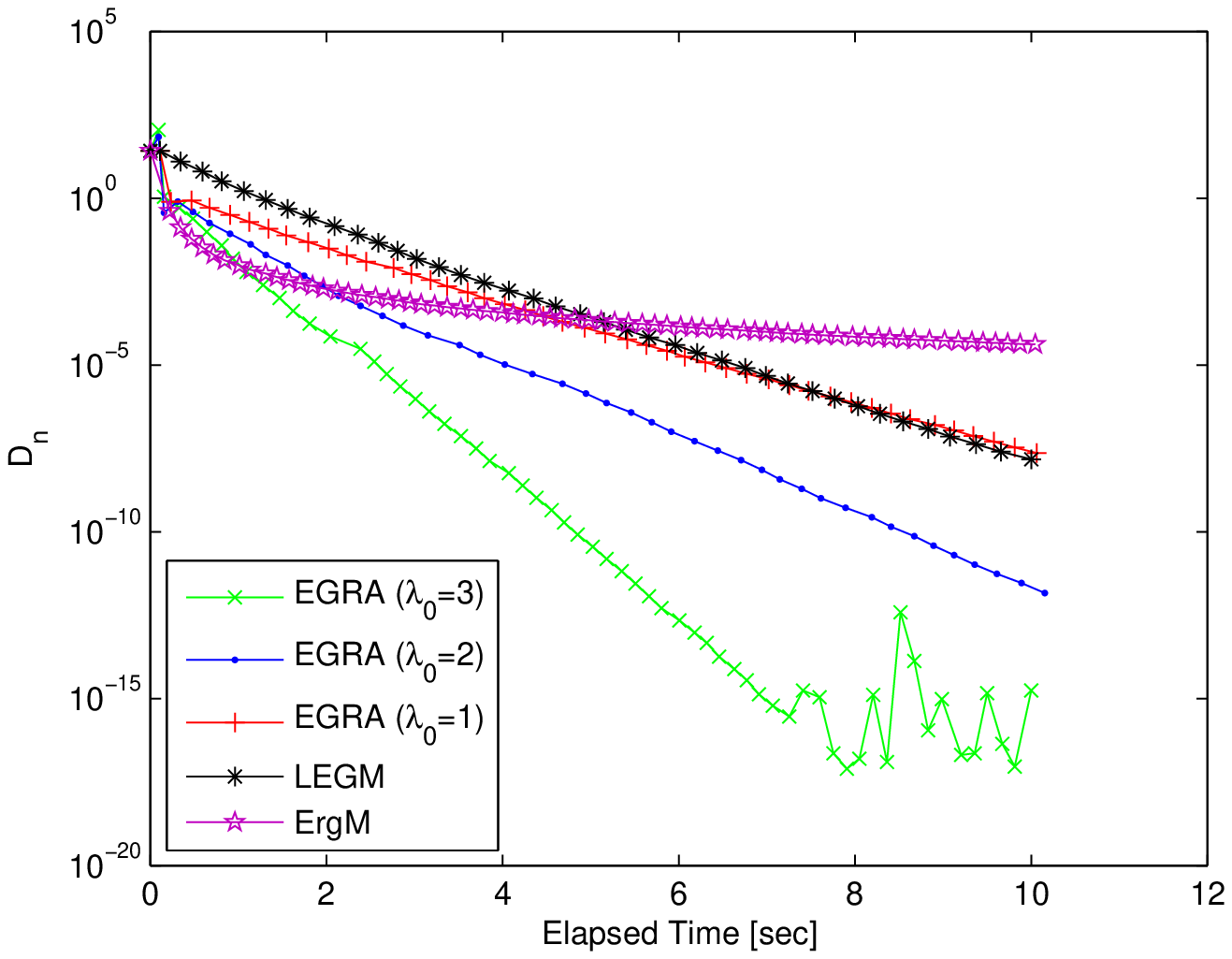}
\caption{$D_n$ and Time (in $\Re^{200}$)}\label{fig3}
\end{minipage}
\hfill
\begin{minipage}[b]{0.45\textwidth}
\centering
\includegraphics[height=5cm,width=6cm]{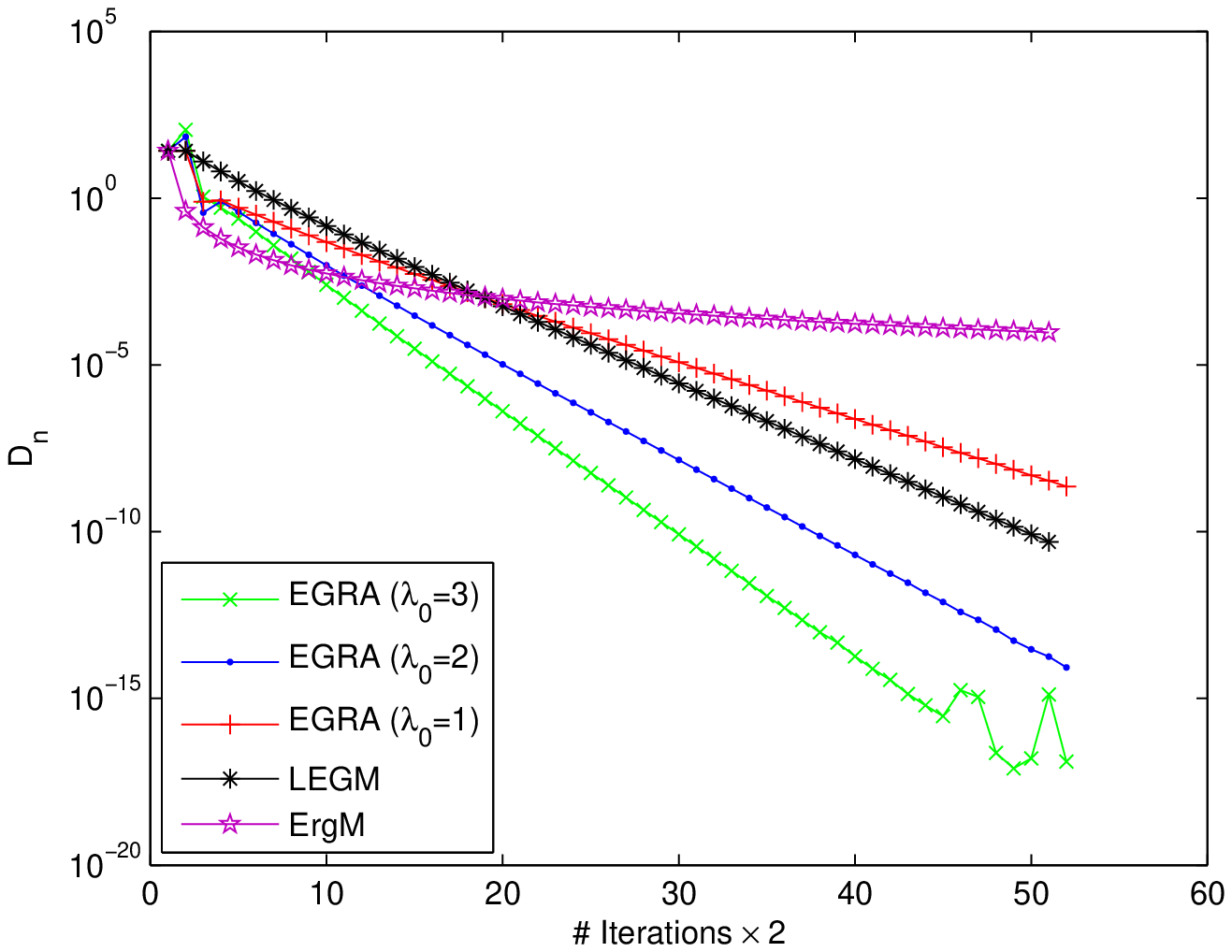}
\caption{$D_n$ and \# Iterations (in $\Re^{200}$)}\label{fig4}
\end{minipage}
\end{figure}

\begin{figure}[!ht]
\begin{minipage}[b]{0.45\textwidth}
\centering
\includegraphics[height=5cm,width=6cm]{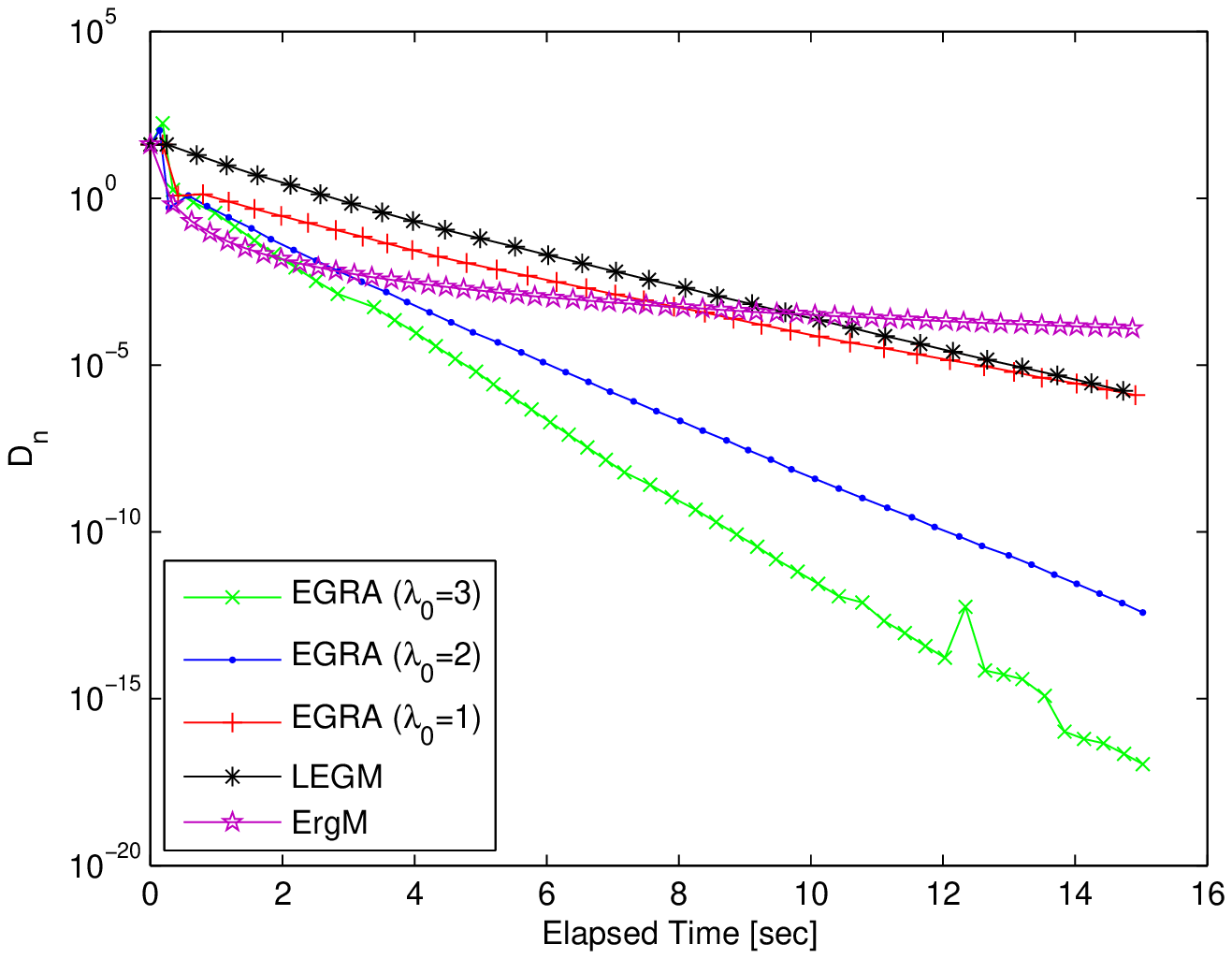}
\caption{$D_n$ and Time (in $\Re^{300}$)}\label{fig5}
\end{minipage}
\hfill
\begin{minipage}[b]{0.45\textwidth}
\centering
\includegraphics[height=5cm,width=6cm]{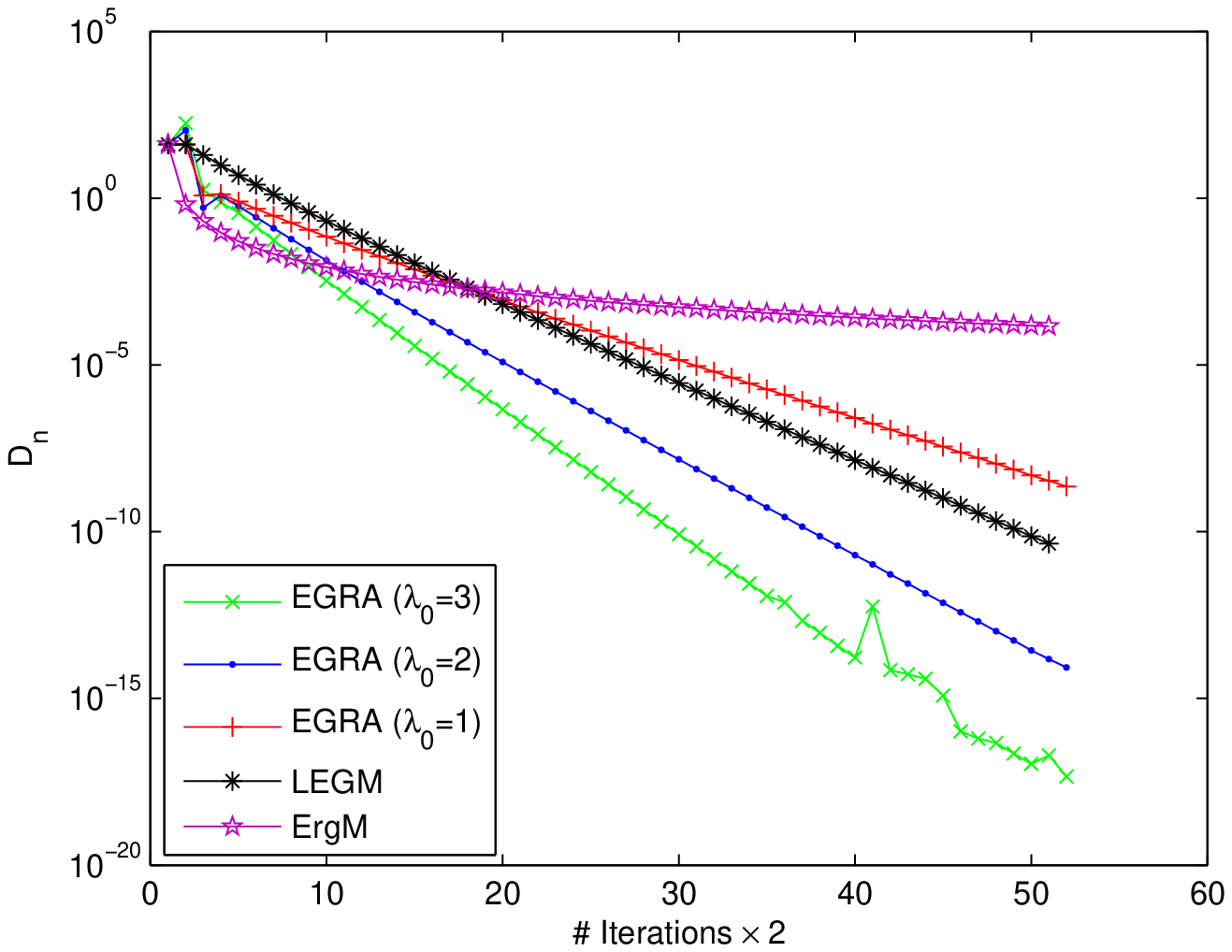}
\caption{$D_n$ and \# Iterations (in $\Re^{300}$)}\label{fig6}
\end{minipage}
\end{figure}
 
\section{Conclusions}
The paper has presented an explicit golden ratio algorithm for solving an equilibrium problem involving a (strongly) pseudomonotone and Lipschitz-type bifunction 
in a finitie dimensional space. The convergence as well as the $R$-linear rate of convergence of the proposed algorithm are established. Several numerical experiments have been done to illustrate the computational advantages of the new algorithm. The main advantages of 
the presented algorithm are that it only requires to solve one optimization problem by iteration and that it does not use any information on the value of the Lipschitz-type constant 
of the bifunction. In addition, the stepsize is generated by the algorithm at each iteration without using any linesearch procedure. 

\section*{Acknowledgements}
We would like to thank \textbf{Dr. Yura Malitsky} and \textbf{Dr. Nguyen The Vinh} for sending us the papers \cite{M2018,V2018}. 
They have given us many valuable comments and suggestions to write this paper.


\begin{thebibliography}{99}
\bibitem{AH2018} Anh, P.N., Hieu, D.V.: Multi-step algorithms for solving EPs. 
Math. Model. Anal. \textbf{23}, 453-472 (2018)

\bibitem{AHT2015}Anh, P.N., Hai, T.N., Tuan, P.M.: On ergodic algorithms for equilibrium problems. 
J. Glob. Optim. \textbf{64}, 179-195 (2016)

\bibitem{BC2011} Bauschke, H. H., Combettes, P. L.: Convex Analysis and Monotone Operator Theory in Hilbert
Spaces. Springer, New York (2011)

\bibitem{BCPP13} Bigi G., Castellani M., Pappalardo M., Pass acantando M. (2013), Existence and solution
metho ds for equilibria, Europ ean Journal of Op erational Research, vol. 227 (1), pp. 1-11

\bibitem{BCPP2019} Bigi, G., Castellani, M., Pappalardo, M., Passacantando, M.: Nonlinear Programming Techniques for Equilibria. 
Springer, Switzerland (2019)

\bibitem{BO1994} Blum, E., Oettli, W.: From optimization and variational inequalities to equilibrium problems.
Math. Student. \textbf{63}, 123--145 (1994)

\bibitem{CKK2004} Contreras, J., Klusch, M., Krawczyk, J. B.: Numerical solutions to Nash-Cournot equilibria in coupled constraint electricity markets. 
IEEE Trans. Power Syst. \textbf{19}, 195-206 (2004)

\bibitem{FP2002} Facchinei, F., Pang, J. S.: Finite-Dimensional Variational Inequalities and Complementarity Problems.
Springer, Berlin (2002)

\bibitem{FA1997} Flam, S. D., Antipin, A. S.: Equilibrium programming and proximal-like algorithms. 
Math. Program. \textbf{78},  29-41 (1997)

\bibitem{HCX2018}Hieu, D. V., Cho, Y. J., Xiao, Y-B.: Modified extragradient algorithms for solving equilibrium problems.
Optimization (2018). DOI: 10.1080/02331934.2018.1505886

\bibitem{H2018MMOR}  Hieu, D. V.: An inertial-like proximal algorithm for equilibrium problems.
Math. Meth. Oper. Res. (2018). DOI:10.1007/s00186-018-0640-6

\bibitem{Hieu2018NUMA}Hieu, D. V.: New inertial algorithm for a class of equilibrium problems.
Numer. Algor. (2018). DOI:10.1007/s11075-018-0532-0

\bibitem{K2003} Konnov, I. V. : Application of the proximal point method to non-monotone equilibrium problems. 
J. Optim. Theory Appl. \textbf{119}, 317-333 (2003)

\bibitem{K2007} Konnov, I. V.: Equilibrium Models and Variational Inequalities. 
Elsevier, Amsterdam (2007)

\bibitem{KP2003} Konnov, I.V., Pinyagina, O.V.: D-gap functions for a class of equilibrium problems in
Banach spaces, Comput. Meth. Appl. Math. \textbf{3}, 274-286 (2003)

\bibitem{K1972} Fan, K.: A minimax inequality and applications. In: Shisha, O. (ed.) Inequality III, pp. 103-113. Academic
Press, New York (1972)

\bibitem{LS2016} Lyashko, S. I. , Semenov, V. V.: A new two-step proximal algorithm of solving the problem of equilibrium programming. 
In Optimization and Its Applications in Control and Data Sciences. Springer, Switzerland \textbf{115}, 315-325 (2016)

\bibitem{M2018} Malitsky, Y.: Golden ratio algorithms for variational inequalities. (2018). https://arxiv.org/abs/1803.08832

\bibitem{M1999} Moudafi, A.: Proximal point algorithm extended to equilibrum problem. 
J. Nat. Geometry, \textbf{15}, 91-100 (1999)

\bibitem{M1984} Muu, L.D.: Stability property of a class of variational inequalities. 
Optimization \textbf{15}, 347-353 (1984)

\bibitem{MO1992}  Muu, L. D., Oettli, W.: Convergence of an adative penalty scheme for finding constrained equilibria. 
Nonlinear Anal. TMA \textbf{18}, 1159--1166 (1992)

\bibitem{MQ15} Muu, L. D., Quy, N.V.: On Existence and solution methods for strongly pseudomonotone equilibrium problems. 
Vietnam J. Math. \textbf{43}, 229-238 (2015)

\bibitem{QMH2008}Quoc, T. D., Muu, L. D., Nguyen, V. H.: Extragradient algorithms extended to equilibrium problems. 
Optimization \textbf{57}, 749-776 (2008)

\bibitem{OR70} J. M. Ortega and W. C. Rheinboldt. Iterative Solution of Nonlinear Equations in Several
Variables, Academic Press, New York, 1970

\bibitem{SS2011} Santos, P., Scheimberg, S.: An inexact subgradient algorithm for equilibrium problems. 
Comput. Appl. Math. \textbf{30}, 91-107 (2011)

\bibitem{SNN2013}Strodiot, J. J., Nguyen, T. T. V., Nguyen, V. H.: A new class of hybrid extragradient algorithms for solving quasi-equilibrium problems. 
J. Glob. Optim. \textbf{56}, 373-397 (2013)

\bibitem{V2018} Vinh, N. T.: Golden ratio algorithms for solving equilibrium problems in Hilbert spaces. (2018). https://arxiv.org/abs/1804.01829
\end{thebibliography}
\end{document}